\documentclass[12pt]{article}
\usepackage{amssymb}
\usepackage{amsfonts}
\usepackage{amsmath}
\usepackage{amsthm}
\usepackage{eucal}
\usepackage{latexsym}
\usepackage{array}
\usepackage{pstricks,pst-node}
\theoremstyle{plain}
\newtheorem{thm}{Theorem}[section]
\newtheorem{prop}{Proposition}[section]
\newtheorem{cor}{Corollary}[section]
\theoremstyle{remark}
\newtheorem{defn}{Definition}[section]

\newtheorem{demo}{Example}
\def\al{\alpha}
\def\be{\beta}
\def\de{\delta}

\def\T{\mathcal T}
\def\la{\lambda}
\def\si{\sigma}

\def\Si{\Sigma}
\def\tilde{\widetilde}
\def\<{\langle}
\def\>{\rangle}
\def\up{\uparrow}
\def\dn{\downarrow}
\def\Zm{\mathbf Z/m\mathbf Z}
\def\Zn{\mathbf Z/n\mathbf Z}
\def\Hom{\operatorname{Hom}}
\def\Aut{\operatorname{Aut}}
\def\Ob{\operatorname{Ob}}

\def\src{\operatorname{src}}
\def\trg{\operatorname{trg}}
\def\im{\operatorname{im}}

\def\A{\mathcal A}
\def\B{\mathcal B}
\def\C{\mathcal C}
\def\D{\mathcal D}
\def\G{\mathcal G}
\def\K{\mathcal K}
\def\M{\mathcal M}
\def\SS{\mathcal S}
\def\T{\mathcal T}

\def\Y{\mathcal Y}
\def\N{\mathbb N}

\def\kk{\Bbbk}

\def\Ne{\mathcal N}

\def\H{\mathfrak H}
\def\P{\mathcal P}

\def\U{\mathfrak U}
\def\W{\mathfrak W}
\def\Z{\mathbf Z}
\def\UU{\mathfrak U\mathfrak U}
\def\SU{\mathfrak S\mathfrak U}
\def\<{\langle}
\def\>{\rangle}

\begin{document}
\title{Updown Categories:  Generating Functions and Universal
Covers}
\author{Michael E. Hoffman\\
\small Dept. of Mathematics, U. S. Naval Academy\\[-0.8ex]
\small Annapolis, MD 21402 USA\\[-0.8ex]
\small \texttt{meh@usna.edu}}
\date{\small July 6, 2012
\small MR Classifications:  Primary 18B35, 06A07; Secondary 57M10, 05A17, 05C05}
\maketitle
%%ABSTRACT%%%%%%%%%%%%%%%%%%%%%%%%%%%%%%%%%%%%%%%%%%%%%%%%%%%%
\begin{abstract}
A poset can be regarded as a category in which there is at most 
one morphism between objects, and such that at most one of 
$\Hom(c,c')$ and $\Hom(c',c)$ is nonempty for $c\ne c'$.  If we keep 
in place the latter axiom but allow for more than one morphism
between objects, we have a sort of generalized poset in which
there are multiplicities attached to the covering relations, and
possibly nontrivial automorphism groups.  We call such a category
an ``updown category.''  In this paper we give a precise definition
of such categories and develop a theory for them.
We also give a detailed account of ten examples, including 
updown categories of integer partitions, integer compositions, 
planar rooted trees, and rooted trees.
\end{abstract}
%%ARTICLE TEXT BEGINS HERE%%%%%%%%%%%%%%%%%%%%%%%%%%%%%%%%%%%%%%%
\section{Introduction}
\label{S:intro}
\par
Suppose we have a collection of combinatorial objects, naturally
graded, so that any object of rank $n$ can be built up in $n$ 
steps from a single object in rank 0.  Further, for any 
object $p$ of rank $n$ and $q$ of rank $n+1$, there are some
number $u(p;q)$ of ways to build up $p$ to make $q$, and $d(p;q)$
ways to break down $q$ to get $p$.  Here $u(p;q)$ and $d(p;q)$
are nonnegative integers, possibly unequal, though we require
that $u(p;q)\ne 0$ if and only if $d(p;q)\ne 0$.  For example
(as in \cite{H2}) our collection might be the set of rooted trees,
with $u(p;q)$ the number of vertices of the rooted tree $p$ to
which a new edge and terminal vertex can be added to get $q$,
and $d(p;q)$ the number of terminal vertices (and incoming edges)
that can be removed from $q$ to get $p$.
\par
We can obtain a natural definition of such a situation by modifying
the categorical definition of a poset.
A poset is usually thought of as a category with at most one
morphism between objects, and at most one of the sets $\Hom(p,q)$
and $\Hom(p,q)$ nonempty when $p\ne cq$.  
If we keep in place the second condition but permit $\Hom(p,q)$
to have more than one element, we allow for multiplicities (if $p\ne q$) 
and automorphisms (if $p=q$).  
If in addition the object set is graded, we call such a category 
(precisely defined in \S\ref{S:ucdef} below) an ``updown category.''  
For an updown category $\C$, there are nonnegative
integers $u(p;q)$ and $d(p;q)$ for $p,q\in\Ob\C$ with $q$ having 
rank one greater than $p$, such that
$$
u(p;q)|\Aut q|=d(p;q)|\Aut p| .
$$
Then the set $\Ob\C$ has a natural graded poset structure, and
the operators $U$ and $D$ on the free vector space 
$\kk(\Ob\C)$ defined by equations 
$$
Up = \sum_{\text{q covers p}} u(p;q)q\quad\text{and}\quad
Dp = \sum_{\text{p covers q}} d(q;p)q .
$$
are adjoint for the inner product on $\kk(\Ob\C)$ given by
$\<p,q\>=|\Aut p|\de_{p,q}$.
\par
For any updown category $\C$, there are associated two generating
functions, defined in $\S\ref{S:ecgf}$:  the object generating function and
the morphism generating function.  If $\C$ is a univalent updown
category (i.e., $u(p;q)=d(p;q)$ for all $p,q\in\Ob\C$), then the former
is the rank-generating function of the graded poset $\Ob\C$.
Computation of these generating functions is facilitated if $\C$
is evenly up-covered (i.e., $\sum_{\text{$q$ covers $p$}}u(p,q)$
depends only on the grade $|p|$ of $p$) or evenly down-covered
($\sum_{\text{$p$ covers $q$}}d(q;p)$ only depends on $|p|$).
\par
Univalent updown categories admit a natural definition of 
universal covers.
In \cite{H1} the author developed a theory of universal covers 
for weighted-relation posets, i.e., ranked posets in which each
covering relation has a single number $n(x,y)$ assigned to it.
The universal cover of a weighted-relation poset $P$ is the
``unfolding'' of $P$ into a usually much larger weighted-relation
poset $\tilde P$, so that the Hasse diagram of $\tilde P$ is 
a tree and all covering relations of $\tilde P$ have multiplicity 1.
Although $\tilde P$ had a natural description in each of the seven examples 
considered in \cite{H1}, the general construction of $\tilde P$ given 
in \cite[Theorem 3.3]{H1} 
was somewhat unsatisfactory since it involved many arbitrary choices.  
In \S\ref{S:uuccs} we show that univalent updown categories are essentially 
``categorified'' weighted-relation posets and give a functorial 
definition of universal covers for them (Theorem \ref{T:ucover} below).
We also give a functorial description of two univalent updown categories
$\C^{\up}$ and $\C^{\dn}$ associated with an updown category $\C$.
\par
In \S\ref{S:ex} we offer ten examples, which encompass all those given in 
\cite{H1}.  
These include updown categories whose objects are the subsets of a 
finite set, monomials, necklaces, integer partitions, integer compositions, 
planar rooted trees, and rooted trees.  For each example we compute
the object and morphism generating functions and describe the associated
covering spaces.
\par
The author takes pleasure in acknowledging the support he has received 
from various institutions during the extended gestation of this paper.
The basic idea was conceived during the academic year 2002-2003, 
when the author was partially supported by the Naval Academy Research 
Council (NARC).
During the following academic year, while on sabbatical leave, 
he wrote a first draft \cite{H3} during a stay at the Max-Planck-Institut f\"ur 
Mathematik (MPIM) in Bonn, and subsequently presented its contents to 
the research group of Prof. Lothar Gerritzen at Ruhr-Universit\"at Bochum
during a visit supported by the German Academic Exchange Service
(DAAD).
The author thanks Prof. Gerritzen and Dr. Ralf Holtkamp for helpful 
discussions.
During the academic years 2004-2005 and 2005-2006, the author 
received partial support from the NARC.
Finally, he again enjoyed the hospitality of the MPIM during 2012.
\section{Updown categories}
\label{S:ucdef}
\par
We begin by defining an updown category.
\begin{defn}
\label{D:updown}
An updown category is a small category $\C$ with a rank functor 
$|\cdot|:\C\to\N$ (where $\N$ is the ordered set of natural numbers 
regarded as a category) such that
\begin{itemize}
\item[A1.]
Each rank $\C_n=\{p\in\Ob\C:|p|=n\}$ is finite.
\item[A2.]
The zeroth rank $\C_0$ consists of a single object $\hat 0$,
and $\Hom(\hat 0,p)$ is nonempty for all objects $p$ of $\C$.
\item[A3.]
For objects $p,p'$ of $\C$, $\Hom(p,p')$ is always finite, and
$\Hom(p,p')=\emptyset$ unless $|p|<|p'|$ or $p=p'$.  In the
latter case, $\Hom(p,p)$ is a group, denoted $\Aut(p)$.
\item[A4.]
Any morphism $p\to p'$, where $|p'|=|p|+k$, factors as a
composition $p=p_0\to p_1\to\dots\to p_k=p'$, where $|p_{i+1}|=
|p_i|+1$;
\item[A5.]
If $|p'|=|p|+1$, the actions of $\Aut(p)$ and $\Aut(p')$ on $\Hom(p,p')$
(by precomposition and postcomposition respectively) are free.
\end{itemize}
\end{defn}
\par
Given an updown category, we can define the multiplicities mentioned
in the introduction as follows.
\begin{defn}
\label{D:mult}
For any two objects $p,p'$ of an updown category $\C$ with $|p'|=|p|+1$, 
define
$$
u(p;p')=\left|\Hom(p,p')/\Aut(p')\right|=
\frac{|\Hom(p,p')|}{|\Aut(p')|}
$$
and
$$
d(p;p')=\left|\Hom(p,p')/\Aut(p)\right|=
\frac{|\Hom(p,p')|}{|\Aut(p)|}.
$$
\end{defn}
It follows immediately from these definitions that
\begin{equation}
u(p;p')|\Aut(p')|=d(p;p')|\Aut(p)| .
\label{E:udaut}
\end{equation}
\par
We note two extreme cases.  First, suppose $\C_n$ is empty for
all $n>0$.  Then $\C$ is essentially the finite group $\Aut\hat 0$.
Second, suppose that every set $\Hom(p,p')$ has at most one element.
Then $\C$ is a graded poset with least element $\hat 0$.
\par
Two important special types of updown categories are defined
as follows.
\begin{defn}
\label{D:unilat}
An updown category $\C$ is univalent if $\Aut(p)$ is trivial for 
all $p\in\Ob\C$.  
An updown category $\C$ is simple if $\Hom(c,c')$ has at most one 
element for all $c,c'\in\Ob\C$, and the factorization in A4 is unique, 
i.e., for $|c'|>|c|$ any $f\in\Hom(c,c')$ has a unique factorization
into morphisms between adjacent ranks.
\end{defn}
Of course simple implies univalent, but not conversely.
A univalent updown category is the ``categorification'' of a 
weighted-relation poset in the sense of \cite{H1}; see \S4 below
for details.
\par
If $\C$ and $\D$ are updown categories, their product $\C\times\D$ is the 
usual one, i.e. $\Ob(\C\times\D)=\Ob\C\times\Ob\D$ and
$$
\Hom_{\C\times\D}((c,d),(c',d'))=\Hom_{\C}(c,c')\times\Hom_{\D}(d,d') .  
$$
The rank is defined on $\C\times\D$ by $|(c,d)|=|c|+|d|$.  
We have the following result.
\begin{prop}
\label{P:prod}
If $\C$ and $\D$ are updown categories, then so is their product
$\C\times\D$.
\end{prop}
\begin{proof}
Since 
$$
(\C\times\D)_n=\coprod_{i+j=n}\C_i\times\D_j ,
$$
axiom A1 is clear; and evidently $\hat 0=(\hat 0_{\C},\hat 0_{\D})$ 
satisfies A2.  Checking A3 is routine, and
for A4 we can combine factorizations
$$
c=c_0\to c_1\to\dots\to c_k=c'\quad\text{and}\quad
d=d_0\to d_1\to\dots\to d_l=d'
$$
into
$$
(c,d)\to (c_1,d)\to\dots\to (c',d)\to (c',d_1)\to\dots\to (c',d') .
$$
Finally, for A5 note that, e.g.,
$$
\Hom((c,d),(c',d))\cong \Hom(c,c')\times\Aut(d) ,
$$
and the action of $\Aut(c,d)\cong\Aut(c)\times\Aut(d)$ on this set
is free if and only if the action of $\Aut(c)$ on $\Hom(c,c')$ is free.
\end{proof}
\par
We note that the product of two univalent updown categories is univalent,
but the product of simple updown categories need not be simple:  see 
Example \ref{Ex:subset} in \S6 below.
\par
We now define a morphism of updown categories.
\begin{defn}
\label{D:morph}
Let $\C,\D$ be updown categories.  A morphism from $\C$ to $\D$ is
a functor $F:\C\to\D$ with $|F(p)|=|p|$ for all $p\in\Ob\C$, and
such that, 
for any $p,q\in\Ob\C$ with $|q|=|p|+1$, the induced maps
$$
\Aut(p)\to \Aut(F(p)),
$$
$$
%\coprod_{\{q':F(q')=F(q)\}} \Aut(q')\backslash \Hom(p,q')/\Aut(p)\to 
%\Aut(F(q))\backslash \Hom(F(p),F(q))/\Aut(F(p))
\coprod_{\{q': F(q')=F(q)\}} \Hom(p,q')/\Aut(p)\to\Hom(F(p),F(q))/\Aut(F(p)),
$$
and
$$
\coprod_{\{q': F(q')=F(q)\}} \Hom(p,q')/\Aut(q')\to\Hom(F(p),F(q))/\Aut(F(q))
$$
are injective.
\end{defn}
\par
We have the following result.
\begin{prop} 
\label{P:simple}
Suppose $F:\C\to\D$ is a morphism of updown categories.
If $\D$ is univalent, then so is $\C$; if $\D$ is simple, then
$\C$ is also simple and $F$ is injective as a function on object
sets.
\end{prop}
\begin{proof}
It follows immediately from Definition \ref{D:morph} that
$\C$ must be univalent when $\D$ is.
Now suppose $\D$ is simple.  
Then $\C$ is univalent, and it follows from Definition \ref{D:morph} 
that the induced function
$$
\coprod_{\{q':F(q')=F(q)\}}\Hom(p,q')\to \Hom(F(p),F(q))
$$
is injective when $|q|=|p|+1$:  but $\Hom(F(p),F(q))$ is (at most)
a one-element set, so $F$ must be injective on object sets and $\Hom(p,q)$ 
can have at most one object.  But then unique factorization of morphisms
in $\C$ follows from that in $\D$, so $\C$ is simple.
\end{proof}
\par
There is a morphism of updown categories
$\C\to\C\times\D$ given by sending $c\in\Ob\C$ to $(c,\hat 0_{\D})$
whenever $\C$ and $\D$ are updown categories; similarly there is
a morphism $\D\to\C\times\D$.  We denote the $n$-fold cartesian
power of $\C$ by $\C^n$. 
\par
Let $\kk$ be a field of characteristic 0, $\kk(\Ob\C)$ the free vector space 
on $\Ob\C$ over $\kk$.  
We now define ``up'' and ``down'' operators on $\kk(\Ob\C)$.
\begin{defn} 
\label{D:udops}
For an updown category $\C$, let $U,D:\kk(\Ob\C)\to\kk(\Ob\C)$
be the the linear operators given by
$$
Up=\sum_{|p'|=|p|+1} u(p;p')p'
$$
and
$$
Dp=\begin{cases} \sum_{|p'|=|p|-1} d(p';p) p',& |p|>0,\\
0, & p=\hat 0,\end{cases}
$$
for all $p\in\Ob\C$.
\end{defn}
\begin{thm} 
\label{T:adjoint}
The operators $U$ and $D$ are adjoint with respect
to the inner product on $\kk(\Ob\C)$ defined by
\begin{equation*}
\<p,p'\>=\begin{cases} |\Aut (p)|,&\text{if $p'=p$},\\
0,&\text{otherwise.}\end{cases}
\end{equation*}
\end{thm}
\begin{proof} Since $\<Up,p'\>=\<p,Dp'\>=0$ unless $|p'|=|p|+1$, it
suffices to consider that case.  Then
$$
\<Up,p'\>=u(p;p')\<p',p'\>=u(p;p')|\Aut(p')|
$$
while 
$$
\<p,Dp'\>=d(p;p')\<p,p\>=d(p;p')|\Aut(p)|,
$$
and the two agree by equation (\ref{E:udaut}).
\end{proof}
\par
Now we extend the definitions of $u(p;p')$ and $d(p;p')$ to
any pair $p,p'\in\Ob\C$ by setting $u(p;p')=d(p;p')=0$ if 
$\Hom(p,p')=\emptyset$ and
\begin{equation}
\label{E:extend}
u(p;p')=\frac{\<U^{|p'|-|p|}(p),p'\>}{|\Aut(p')|},
\quad
d(p;p')=\frac{\<U^{|p'|-|p|}(p),p'\>}{|\Aut(p)|}
\end{equation}
otherwise.  It is immediate that equation (\ref{E:udaut}) still holds,
and that
$$
U^k(p)=\sum_{|p'|=|p|+k} u(p;p') p'
$$
and 
$$
D^k(p)=\sum_{|p'|=|p|-k} d(p;p')p'
$$
for any $p\in\Ob\C$.  (However, it is no longer true that
$u(p;q)$ and $d(p;q)$ have any simple relation to $|\Hom(p,q)|$
when $|q|-|p|>1$.)  An important special case of the extended
equation (\ref{E:udaut}) is
\begin{equation}
\frac{d(\hat 0;p)}{u(\hat 0;p)}=
\frac{|\Aut(p)|}{|\Aut\hat 0|} 
\label{E:ratio}
\end{equation}
for any object $p$ of $\C$.  If $\Aut\hat 0$ is trivial,
equation (\ref{E:ratio}) gives the order
of the automorphism group of $p\in\Ob\C$ as a ratio of multiplicities
(cf. Proposition 2.6 of \cite{H2}).
We also have the following result.
\begin{thm}
\label{T:trans}
If $|p|\le k\le |q|$, then
$$
u(p;q)=\sum_{|p'|=k} u(p;p')u(p';q),
$$
and similarly for $u$ replaced by $d$.
\end{thm}
\begin{proof} Using equation (\ref{E:extend}), we can write $u(p,q)$ as
\begin{multline*}
\frac{\<U^{|q|-|p|}p,q\>}{|\Aut(q)|}
=\frac1{|\Aut(q)|}\<U^{k-|p|}U^{|q|-k}(p),q\>
=\frac1{|\Aut(q)|}\sum_{|p'|=k}u(p;p')\<U^{k-|p|}p',q\>\\
=\frac1{|\Aut(q)|}\sum_{|p'|=k}u(p;p')u(p';q)|\Aut(q)|
=\sum_{|p'|=k}u(p;p')u(p';q) ,
\end{multline*}
and the proof for $d$ is similar.
\end{proof}
\begin{defn}
\label{D:po}
For an updown category $\C$, define the induced partial order on $\Ob\C$ by 
setting $p\preceq q$ if and only if $\Hom(p,q)\ne\emptyset$.  
\end{defn}
It follows from Theorem \ref{T:trans} that $p\preceq q\iff u(p;q)\ne 0 
\iff d(p;q)\ne 0$.
Henceforth we write $p \lhd q$ if $q$ covers $p$ in the induced partial 
order.  Of course different updown categories can have the same induced
poset:  see Examples \ref{Ex:upart} and \ref{Ex:kpart} of \S5 below.  
If the updown category $\C$
was a poset to start with (thought of as a category in the usual way),
then all the weights $u(p;q)=d(p;q)$ assigned to the covering relations
in $(\Ob\C,\preceq)$ are 1.  We call such an updown category unital.
Evidently simple $\implies$ unital $\implies$ univalent.
\par
In the univalent case, equation (\ref{E:ratio}) is trivial since
$u(p;q)=d(p;q)$ for all $p$ and $q$.  Nevertheless, we have the
following interpretation of the multiplicity in this case.
\begin{thm}
\label{T:count}
Let $\C$ be a univalent updown category.  For $p,q\in\Ob\C$ with
$|q|-|p|=n>0$, $u(p;q)=d(p;q)$ is the number of distinct
strings $(h_1,\dots,h_n)$ such that each $h_i$ is 
a morphism between adjacent ranks and $h_nh_{n-1}\cdots h_1$
is a morphism from $p$ to $q$.
\end{thm}
\begin{proof} We use induction on $n$.  The result is immediate
if $n=1$, since in a univalent updown category
$$
u(p;q)=d(p;q)=|\Hom(p,q)|
$$
when $|q|=|p|+1$.  Now if $N(p,q)$ denotes the number of
strings $(h_1,\dots,h_n)$ as in the statement of the proposition,
it is evident that, for $|q|>|p|+1$, 
$$
N(p,q)=\sum_{r\lhd q} N(p,r)N(r,q) .
$$
But then the inductive step follows from Theorem \ref{T:trans}.
\end{proof}
\par
\section{Even Covering and Generating Functions}
\label{S:ecgf}
In this section we introduce the even covering properties, which
are satisfied in many of the examples of updown categories given in
\S6 below.  We also define several generating functions associated 
with any updown category.
\begin{defn}
\label{D:udcover}
Let $\C$ be an updown category.  Then $\C$ is evenly up-covered if
there is a sequence of numbers $u_0,u_1,\dots$ so that, for any $p\in\C_n$,
$$
\sum_{q \rhd p} u(p;q)=u_n .
$$
Dually, $\C$ is evenly down-covered if there is a sequence of numbers 
$d_1,d_2,\dots$ so that, for any $p\in\C_n$,
$$
\sum_{q \lhd p} d(q;p)=d_n .
$$
\end{defn}
We note that any simple updown category is evenly down-covered with 
$d_n=1$ for all $n$.  Another special case occurs often enough that
we make the following definition.
\begin{defn}
We call $\C$ factorial if it is evenly down-covered with
$$
\sum_{q \lhd p} d(q;p)=|p| 
$$
for all $p\in\Ob\C$.
\end{defn}
\par
If $\C$ is evenly up-covered, then by induction using Theorem \ref{T:trans}
it follows that
\begin{equation}
\label{E:ufact}
\sum_{|c|=n} u(\hat0;c)=u_0u_1\cdots u_{n-1} 
\end{equation}
for $c\in\C_n$.
On the other hand, if $\C$ is evenly down-covered then one has
$D^nc=d_nd_{n-1}\cdots d_1\hat0$ for $c\in\C_n$, and so
\begin{equation}
\label{E:dfact}
d(\hat0;c)=\frac{\<\hat0,D^nc\>}{|\Aut\hat0|}=d_nd_{n-1}\cdots d_1
\end{equation}
for such $c$.  In particular, $d(\hat0;c)=|c|!$ for all $c\in\Ob\C$
if $\C$ is factorial.
\par
Although the even covering properties are not generally preserved under
products, we do have the following result.
\begin{thm}
\label{T:fprod}
If $\C$ and $\D$ are factorial updown categories, then
so is $\C\times\D$.
\end{thm}
\begin{proof} Any object covered by $(c,d)\in\Ob(\C\times\D)$ must have
the form $(c',d)$ with $c'\lhd c$, or $(c,d')$, with $d'\lhd d$.
Thus
\begin{multline*}
\sum_{p\lhd (c,d)} d(p;(c,d))
=\sum_{c'\lhd c}d((c',d);(c,d))+\sum_{d'\lhd d}d((c,d');(c,d))\\
=\sum_{c'\lhd c}d(c';c)+\sum_{d'\lhd d}d(d';d)
=|c|+|d|=|(c,d)|.
\end{multline*}
\end{proof}
Neither of the two even covering properties implies the other.  
Examples \ref{Ex:plrtree} and \ref{Ex:brtree} of \S5 below are 
evenly up-covered but not
evenly down-covered, and it is easy to construct simple updown
categories that are not evenly up-covered.  For simple updown
categories that are evenly up-covered, we have the following result.
\begin{thm}
\label{T:unique}
Suppose $\C$ and $\D$ are simple updown categories that
are both evenly up-covered with the same sequence $\{u_n\}$.  Then 
$\C$ and $\D$ are isomorphic as updown categories.
\end{thm}
\begin{proof}
It suffices to give a functor $F:\C\to\D$ that is bijective on the
object sets such that $F(c)\lhd F(c')$ for all $c,c'\in\Ob\C$.  We 
proceed by induction on rank.  For rank 0 we set $F(\hat 0_{\C})=
\hat 0_{\D}$.  Suppose $F$ has been defined through rank $n$.
For each $p\in\C_n$, choose a bijection $\phi_p$ from $C^+(p)=
\{p'\in\C_{n+1}|p\lhd p'\}$ to $C^+(F(p))$ (which is possible since
both sets have $u_n$ elements).  Then for $q\in\C_{n+1}$, set
$F(q)=\phi_p(q)$, where $p$ is the unique element of $\C_n$ that
$q$ covers.
\end{proof}
Now we turn to generating functions.
\begin{defn} 
\label{D:gf}
Let $\C$ be an updown category.  The object
generating function of $\C$ is
$$
O_{\C}(t)=\sum_{p\in\Ob\C}\frac{t^{|p|}}{|\Aut(p)|}=
\sum_{n\ge 0}\sum_{p\in\C_n} \frac{t^n}{|\Aut(p)|} ,
$$
and the morphism generating function of $\C$ is
\begin{equation}
\label{E:mgf1}
M_{\C}(t)=\sum_{p,q\in\Ob\C,\ p\lhd q}\frac{u(p;q)t^{|p|+|q|}}
{|\Aut(p)|}=
\sum_{n\ge 0}\sum_{p\in\C_n}\sum_{q\in\C_{n+1}} \frac{u(p;q) t^{2n+1}}
{|\Aut(p)|} .
\end{equation}
\end{defn}
Both $O_{\C}(t)$ and $M_{\C}(t)$ are elements of the formal power
series ring $\mathbb Q[[t]]$.  If $\C$ is univalent, then
$$
O_{\C}(t)=\sum_{n\ge 0}|\C_n| t^n
$$
and
$$
M_{\C}(t)=\sum_{n\ge 0}\sum_{p\in\C_n}\sum_{q\in\C_{n+1}}|\Hom(p,q)|t^{2n+1}
$$
are elements of $\mathbb Z[[t]]$.
In view of equation (\ref{E:udaut}), the morphism generating function
can be written
\begin{equation}
\label{E:mgf2}
M_{\C}(t)=\sum_{p,q\in\Ob\C,\ p \lhd q}\frac{d(p;q)t^{|p|+|q|}}
{|\Aut(q)|} .
\end{equation}
\begin{defn}
\label{D:fs}
For an updown category $\C$, the formal series of $\C$ is
$$
S_{\C}(t)=\sum_{p\in\Ob\C}\frac{p t^{|p|}}{|\Aut(p)|}\in \kk(\Ob\C)[[t]] .
$$
\end{defn}
These definitions are related by the following result.
\begin{thm} 
\label{T:gfrelation}
If the inner product $\<,\>$ of Theorem \ref{T:adjoint} is extended to
$\kk(\Ob\C)[[t]]$, then
\begin{equation}
\<S_{\C}(t),S_{\C}(t)\>=O_{\C}(t^2)
\label{E:SS}
\end{equation}
and 
\begin{equation}
\<US_{\C}(t),S_{\C}(t)\>=\<S_{\C}(t),DS_{\C}(t)\>=M_{\C}(t).
\label{E:USS}
\end{equation}
\end{thm}
\begin{proof} Immediate from Theorem \ref{T:adjoint} and the definitions.
\end{proof}
\par
The generating functions of a product can be obtained from those of
its factors as follows.
\begin{cor}
\label{C:gfprod}
For updown categories $\C$ and $\D$, 
\begin{equation}
O_{\C\times\D}(t)=O_{\C}(t)O_{\D}(t)
\label{E:oprod}
\end{equation}
and
\begin{equation}
M_{\C\times\D}(t)=M_{\C}(t)O_{\D}(t^2)+O_{\C}(t^2)M_{\D}(t) .
\label{E:mprod}
\end{equation}
\end{cor}
\begin{proof} We have $S_{\C\times\D}(t)=S_{\C}(t)\otimes S_{\D}(t)$
under the evident identification of $\kk(\Ob(\C\times\D))$
with $\kk(\Ob\C)\otimes\kk(\Ob\D)$.
Hence
\[
\<S_{\C\times\D}(t),S_{\C\times\D}(t)\>=\<S_{\C}(t)\otimes S_{\D}(t),
S_{\C}(t)\otimes S_{\D}(t)\>=
\<S_{\C}(t),S_{\C}(t)\>\<S_{\D}(t),S_{\D}(t)\>,
\]
and equation (\ref{E:oprod}) follows using equation (\ref{E:SS}).
Similarly, we have
\begin{multline*}
\<US_{\C\times\D}(t),S_{\C\times\D}(t)\>=
\<U(S_{\C}(t)\otimes S_{\D}(t)),S_{\C}(t)\otimes S_{\D}(t)\>=\\
\<US_{\C}(t)\otimes S_{\D}(t)+S_{\C}(t)\otimes US_{\D}(t),
S_{\C}(t)\otimes S_{\D}(t)\>=\\
\<US_{\C}(t),S_{\C}(t)\>\<S_{\D}(t),S_{\D}(t)\>+
\<S_{\C}(t),S_{\C}(t)\>\<US_{\D}(t),S_{\D}(t)\>
\end{multline*}
from which equation (\ref{E:mprod}) follows via equation (\ref{E:USS}).
\end{proof}
\emph{Remark.}
It follows from the preceding result that
\[
O_{\C^n}(t)=(O_{\C}(t))^n\quad\text{and}\quad
M_{\C^n}(t)=nM_{\C}(t)(O_{\C}(t^2))^{n-1}
\]
where $\C^n$ is the $n$-fold product of $\C$.  
\par
If $\C$ is evenly up-covered or evenly down-covered, there
is a direct relation between the object and morphism generating functions.  
\begin{thm}
\label{T:gfsum}
Let $\C$ be an updown category with 
$$
O_{\C}(t)=\sum_{n\ge 0} a_nt^n .
$$
\begin{itemize}
\item[1.] If $\C$ is evenly up-covered, then
$$
M_{\C}(t)=\sum_{n\ge 0} a_nu_n t^{2n+1} .
$$
\item[2.] If $\C$ is evenly down-covered, then
$$
M_{\C}(t)=\sum_{n\ge 1} a_nd_n t^{2n-1} .
$$
\end{itemize}
\end{thm}
\begin{proof} Immediate from equations (\ref{E:mgf1}) and (\ref{E:mgf2})
respectively.
\end{proof}
\emph{Remark.} Two consequences of the second part are:
(i) if $\C$ is simple, then $O_{\C}(t^2)=1+tM_{\C}(t)$; and
(ii) if $\C$ is factorial, then $M_{\C}(t)=tO_{\C}'(t^2)$.
\par
If the updown category $\C$ is both evenly up-covered and evenly 
down-covered, the preceding result gives two expressions for 
$M_{\C}(t)$ which must agree.  This gives us the following 
result.
\begin{cor}
\label{C:eucedc}
Suppose the updown category $\C$ is both evenly up-covered (with
sequence $\{u_n\}$) and evenly down-covered (with sequence $\{d_n\}$).
Then $a_nu_n=a_{n+1}d_{n+1}$ for all $n\ge 0$, where
$O_{\C}(t)=\sum_{n\ge 0} a_nt^n$.  In particular, if $\C$ is
evenly up-covered and factorial, then $a_0=|\Aut\hat 0|^{-1}$ and
$$
a_n=\frac{u_0u_1\cdots u_{n-1}}{n!|\Aut\hat 0|} ,\quad n\ge 1.
$$
\end{cor}
\section{Univalent Updown Categories, Weighted-relation Posets, and 
Universal Covers}
\label{S:uuccs}
Let $\U$ be the category of updown categories, $\UU$ the full subcategory
of univalent updown categories.  For a functor $F$  between univalent 
updown categories $\C$, $\D$, Definition \ref{D:morph} reduces to the 
requirement that $F$ preserve rank and that the induced function
\begin{equation}
\label{E:moruu}
\coprod_{\{q':F(q')=F(q)\}}\Hom(p,q')\to\Hom(F(p),F(q))
\end{equation}
be injective whenever $p,q\in\Ob\C$ with $|q|=|p|+1$.
\par
The notion of a weighted-relation poset was defined in \cite{H1}.
This consists of a ranked poset
$$
P=\bigcup_{n\ge 0} P_n
$$
with a least element $\hat 0\in P_0$, together with nonnegative
integers $n(x,y)$ for each $x,y\in P$ so that $n(x,y)=0$ unless
$x\preceq y$, and
\begin{equation}
n(x,y)=\sum_{|z|=k}n(x,z)n(z,y)
\label{E:wrp}
\end{equation}
whenever $|x|\le k\le |y|$.
A morphism of weighted-relation posets $P,Q$ is a rank-preserving map 
$f:P\to Q$ such that
\begin{equation}
n(f(t),f(s))\ge \sum_{s'\in f^{-1}(f(s))} n(t,s')
\label{E:mwrp}
\end{equation}
for any $s,t\in P$ with $|s|=|t|+1$.  
Let $\W$ be the category of weighted-relation posets.
\par
Given an updown category $\C$, it follows from Theorem \ref{T:trans}
that the weight functions $n(x,y)=u(x;y)$ and $n(x,y)=d(x;y)$
on the poset $\Ob\C$ (with the partial order defined by Definition
\ref{D:po}) both satisfy equation (\ref{E:wrp}).  So we have two 
weighted-relation posets based on $\Ob\C$ corresponding
to these two sets of weights.  In fact, we can describe
them functorially.
\par
If $\C$ is an updown category, we can form a univalent updown 
category $\C^\up$ with $\Ob\C^\up=\Ob\C$, and with $\Hom_{\C^\up}(p,p')$ 
defined as follows.
We declare $\Hom_{\C^\up}(p,p)=\Aut_{\C^\up}(p)$ trivial for all $p$,
and for $|p'|>|p|$ define $\Hom_{\C^\up}(p,p')$ as 
the set of equivalence classes in $\Hom_{\C}(p,p')$ under the relation
$f_n f_{n-1}\cdots f_1\sim\al_n f_n \cdots \al_1 f_1$, where
each $f_i$ is a morphism between adjacent ranks and $\al_i\in\Aut(\trg f_i)$. 
It is routine to check that $\C^\up$ satisfies the axioms of an updown
category, and for $p,p'\in\Ob\C$ with $|p'|=|p|+1$ the multiplicity is
$$
|\Hom_{\C^\up}(p,p')|=\left|\Hom_{\C}(p,p')/\Aut_{\C}(p')\right|=u(p;p') .
$$
Of course $\C^\up$ coincides with $\C$ if $\C$ is univalent.
\par
Similarly, for any updown category $\C$ there is a univalent updown
category $\C^\dn$ with $\Ob\C^\dn=\Ob\C$, trivial automorphisms, and
$\Hom_{\C^\dn}(p,p')$ defined as the set of equivalence classes 
in $\Hom_{\C}(p,p')$ under the relation $f\sim f_n\be_nf_{n-1}\cdots f_1\be_1$
for $f=f_nf_{n-1}\cdots f_1$ a factorization of $f\in\Hom_{\C}(p,p')$ 
into morphisms between adjacent ranks and $\be_i\in\Aut(\src f_i)$.
Then 
$$
|\Hom_{\C^\dn}(p,p')|=\left|\Hom_{\C}(p,p')/\Aut_{\C}(p)\right|=d(p;p')
$$
for $p,p'\in\Ob\C$ with $|p'|=|p|+1$.
We have the following result.
\begin{thm}
\label{T:updnfunc}
There are two functors $\U\to\UU$, taking an updown
category $\C$ to $\C^\up$ and $\C^\dn$ respectively.
\end{thm}
\begin{proof}
We first consider the ``up'' functor.  For a morphism $F:\C\to\D$
of updown categories, there is an induced functor $F^\up:\C^\up\to\D^\up$
of univalent updown categories: $F^\up(p)=F(p)$ for $p\in\Ob\C$, and
$F^{\up}$ sends the equivalence class $[f]$, where $f\in\Hom_{\C}(p,q)$,
to the equivalence class $[F(f)]\in\Hom_{\D^\up}(F(p),F(q))$.
Now Definition \ref{D:morph} requires that $F$ preserve rank and that 
the induced function
$$
\coprod_{\{q':F(q')=F(q)\}}\Hom_{\C}(p,q')/\Aut_{\C}(p')\to 
\Hom_{\D}(F(p),F(q))/\Aut_{\D}(F(q))
$$
be injective for all $p,q\in\Ob\C$ with $|q|=|p|+1$.  This is
exactly the statement that the induced functor $F^\up$ is a
morphism of univalent updown categories.
The proof for the ``down'' functor is similar.
\end{proof}
\par
Note that the functors of the preceding result respect products, 
e.g., $(\C\times\D)^\up$ can be naturally identified with $\C^\up\times\D^\up$.
Note also that $\C^\up$ is evenly up-covered if $\C$ is, and $\C^\dn$ is
evenly-down covered if $\C$ is.
Now we pass from univalent updown categories to weighted-relation posets.
\begin{thm}
\label{P:wrpfunc}
There is a functor $Wrp:\UU\to\W$, sending
a univalent updown category $\C$ to the set $\Ob\C$ with
the partial order of Definition \ref{D:po} and the
weight function $n(x,y)=u(x;y)=d(x;y)$.
\end{thm}
\begin{proof} The only thing to check is the morphisms.  Suppose
$F:\C\to\D$ is a morphism of $\UU$.  Then $F$ defines a function 
on the object sets, and the function (\ref{E:moruu}) is injective.
Hence
$$
\sum_{\{q':F(q')=F(q)\}}|\Hom(p,q')|\le |\Hom(F(p),F(q))|
$$
and so (since, e.g., $n(p,q')=|\Hom(p,q')|$), inequality
(\ref{E:mwrp}) holds and $F$ induces a morphism of weighted-relation
posets.
\end{proof}
\par
As defined in \cite{H1},
a morphism $f:P\to Q$ of weighted-relation posets is a covering map 
if $f$ is surjective and the inequality (\ref{E:mwrp}) is an equality.  
A universal cover $\tilde P$ of $P$ is a cover $\tilde P\to P$ such 
that, if $P'\to P$ is any other cover, then there is a covering 
map $\tilde P\to P'$ so that the composition $\tilde P\to P'\to P$
is the cover $\tilde P\to P$.
In \cite{H1} such a universal cover was constructed for any 
weighted-relation poset $P$.
\par
In fact, the construction of \cite{H1} can be made considerably
simpler and more natural if we work instead with univalent updown
categories.  We first categorify the definition of covering map.
\begin{defn}
\label{D:cover}
A morphism $\pi:\C'\to\C$ of univalent updown categories is a 
covering map if $\pi$ is surjective on the object sets and the 
induced function
\begin{equation}
\coprod_{\{q':\pi(q')=\pi(q)\}}\Hom(p,q')\to\Hom(\pi(p),\pi(q))
\label{E:cover}
\end{equation}
is a bijection for all $p,q\in\Ob\C'$ with $|q|=|p|+1$.  A covering
map $\pi:\tilde\C\to\C$ is universal if for any other covering map
$\phi:\C'\to\C$ there is a covering map $\psi:\tilde\C\to\C'$ with
$\pi=\phi\psi$.
\end{defn}
Then we have the following result.
\begin{thm}
\label{T:ucover}
Every univalent updown category $\C$ has a universal cover $\tilde\C$.
\end{thm}
\begin{proof}
We define $\tilde\C$ to be the category whose rank-$n$ objects are
strings $(f_1,f_2,\dots,f_n)$ of morphisms $f_i\in\Hom(c_{i-1},c_i)$,
where $c_i\in\C_i$, and whose morphisms are just inclusions of
strings.  It is straightforward to verify that $\tilde\C$ is a
univalent updown category (with $\hat 0_{\tilde\C}$ the empty string).  
Define the functor $\pi:\tilde\C\to\C$ by sending the empty
string to $\hat 0\in\Ob\C$, the nonempty string
$(f_1,\dots,f_n)$ of $\tilde\C$ to the target of $f_n$ in $\Ob\C$,
and the inclusion $(f_1,\dots,f_j)\subset (f_1,\dots,f_n)$ to
the morphism $f_nf_{n-1}\cdots f_{j+1}\in\Hom(c_j,c_n)$.
That the induced function (\ref{E:cover}) is a bijection is a tautology.
\par
Now let $P:\C'\to\C$ be another cover of $\C$.  To define a covering map 
$F:\tilde\C\to\C'$ with $\pi=PF$, we proceed by induction on rank.  
Start by sending the empty string in $\tilde\C_0$ to the element 
$\hat 0$ of $\C'$.  
Now suppose $F$ is defined through rank $n-1$, and
consider a rank-$n$ object $(f_1,\dots,f_n)$ of $\tilde\C$.
Let $c_n=\pi(f_1,\dots,f_n)$.  By the induction hypothesis we
have $c_{n-1}'=F(f_1,\dots,f_{n-1})\in\Ob\C'$, and $c_{n-1}=P(c_{n-1}')$
is the target of $f_{n-1}$, hence the source of $f_n$.
Since
$$
P:\coprod_{\{c':p(c')=c_n\}}\Hom(c_{n-1}',c')\to\Hom(c_{n-1},c_n)
$$
is a bijection, there is a unique morphism $g$ of $\C'$ with
$\src(g)=c_{n-1}'$ sent to $f_n:c_{n-1}\to c_n$.  We define $F(f_1,\dots,f_n)$
to be $\trg(g)$, and the image of the inclusion of $(f_1,\dots,f_{n-1})$
in $(f_1,\dots,f_n)$ to be $g$.  This actually defines the functor
$F$ through rank $n$, since by the induction hypothesis $F$ assigns
to the inclusion of any proper substring $(f_1,\dots,f_k)$ in
$(f_1,\dots,f_{n-1})$ a morphism $h$ from $F(f_1,\dots,f_k)$ to $c_{n-1}'$
in $\C'$; then $F$ sends the inclusion of $(f_1,\dots,f_k)$ in 
$(f_1,\dots,f_n)$ to $gh$.
\end{proof}
\par
\emph{Remark.}
Thinking of the functor $\pi:\tilde\C\to\C$ as a function on the object sets, 
the number of elements of $\tilde\C_n$ which $\pi$ sends to $p\in\C_n$ is 
\begin{equation}
\label{E:fiber}
|\pi^{-1}(p)|=u(\hat 0;p)=d(\hat 0;p)=\frac{\<\hat 0, D^np\>}{|\Aut\hat 0|} ,
\end{equation}
as follows from Theorem \ref{T:count}.
In particular, if $\C$ is factorial then $|\pi^{-1}(p)|=n!$ and
$|\tilde\C_n|=n!|\C_n|$.
Also, it follows from equation (\ref{E:fiber}) that
\begin{equation}
\label{E:coversum}
|\tilde\C_n|=\sum_{p\in\C_n}u(\hat0;p) .
\end{equation}
If $\C$ is evenly up-covered, this equation and equation
(\ref{E:ufact}) imply that
\begin{equation}
\label{E:covereuc}
|\tilde\C_n|=u_0u_1\cdots u_{n-1} .
\end{equation}
\par
The construction of $\tilde\C$ in the the proof of Theorem
\ref{T:ucover} is functorial:  given a morphism $F:\C\to\D$ of
univalent updown categories, we have a morphism $\tilde F:\tilde\C\to
\tilde\D$ given by
$$
\tilde F(f_1,f_2,\dots,f_n)=(F(f_1),F(f_2),\dots,F(f_n)).
$$
Also, the updown category $\tilde\C$ is evidently simple.
Thus, if $\SU$ is the full subcategory of simple updown
categories in $\U$, then there is a functor $\UU\to\SU$
taking $\C$ to $\tilde\C$.  In fact, we have the following
result.
\begin{prop}
\label{P:rtadj}
The functor $\UU\to\SU$ taking $\C$ to $\tilde\C$
is right adjoint to the inclusion functor $\SU\to\UU$.
\end{prop}
\begin{proof} It suffices to show that
$$
\Hom_{\UU}(\C,\D)\cong \Hom_{\SU}(\C,\tilde\D)
$$
for any simple updown category $\C$ and univalent updown category
$\D$.  A morphism $F:\C\to\D$ of univalent
updown categories gives rise to $\tilde F:\tilde\C\to\tilde\D$,
and since $\C$ is simple there is a natural identification 
$\C\cong\tilde\C$, giving us a morphism $\C\to\tilde\D$.  To
go back the other way, just compose with the covering map
$\pi:\tilde\D\to\D$.
\end{proof}
The universal cover functor $\UU\to\SU$ does not respect products:
in fact, $\tilde\C\times\tilde\D$ is generally not simple.
(This does not contradict the preceding result, because our product
is not a categorical product in $\UU$.)  We do have the following result.
\begin{prop}
\label{P:produncvr}
If $\C$ and $\D$ are univalent updown categories, then the 
number of rank-$n$ objects in $\tilde{\C\times\D}$ is
$$
\sum_{k=0}^n \binom{n}{k}|\tilde\C_k||\tilde\D_{n-k}| .
$$
\end{prop}
\begin{proof}
Using the generating functions of the preceding section, equation
(\ref{E:coversum}) can be written
\[
O_{\tilde\C}(t^2)=\<(1-tU)^{-1}\hat0,S_{\C}(t)\> .
\]
Then
\begin{align*}
O_{\tilde{\C\times\D}}(t^2)=&\<(1-tU)^{-1}(\hat0_{\C}\otimes\hat0_{\D}),
S_{\C}(t)\otimes S_{\D}(t)\>\\
=&\sum_{n\ge 0} t^n\sum_{k=0}^n\binom{n}{k}\<U^k\hat0_{\C}\otimes U^{n-k}
\hat0_{\D},S_{\C}(t)\otimes S_{\D}(t)\>\\
=&\sum_{n\ge 0} t^n\sum_{k=0}^n\binom{n}{k}\<U^k\hat0_{\C},S_{\C}(t)\>
\<U^{n-k}\hat0_{\D},S_{\D}(t)\>\\
=&\sum_{n\ge 0} t^n\sum_{k=0}^n\binom{n}{k}t^k|\tilde\C_k|
t^{n-k}|\tilde\D_{n-k}|,
\end{align*}
from which the conclusion follows.
\end{proof}
\section{Examples}
\label{S:ex}
\par
In this section we present ten examples of updown categories.  
Many of the associated weighted-relation posets appear in the last
section of \cite{H1}.  For the convenience of the reader we
have included a cross-reference to \cite{H1} at the beginning
of each example where it applies.
\par
\begin{demo}
\label{Ex:subset}\rm
(Subsets of a finite set; \cite[Ex. 1]{H1}, \cite[Ex. 2.5(b)]{S2},
\cite[Ex. 6.2.6]{E}.)
First, let $\A$ be an updown category such that $\A_0=\{\hat 0\}$,
$\A_1=\{\hat 1\}$, $\A_n=\emptyset$
for $n\ne 0,1$, and $\Hom(\hat 0,\hat 1)$ has a single element.
The groups $\Aut(\hat 0)$ and $\Aut(\hat 1)$ are trivial since
they act freely on the one-element set $\Hom(\hat 0,\hat 1)$.
The object and morphism generating functions are evidently
$$
O_{\A}(t)=1+t\quad\text{and}\quad M_{\A}(t)=t .
$$
Evidently $\A$ is simple and factorial.
\par 
Now let $\B=\A^n$.  Since $\A$ is factorial, $\B$ is factorial by 
Theorem \ref{T:fprod}.
Objects of $\B$ can be identified with subsets of $\{1,2,\dots,n\}$:  
an $n$-tuple $(c_1,\dots,c_n)$ corresponds to the set $\{i: c_i=\hat 1\}$.  
The induced partial order is inclusion of sets, and in fact $\B$
is unital (but not simple for $n\ge 2$).
In \cite{H1} it is shown that the universal cover $\tilde\B$ is the 
simple updown category whose rank-$m$ elements are linearly ordered 
$m$-element subsets of $\{1,\dots,n\}$, and whose morphisms are
inclusions of initial segments.  This makes it obvious that 
$|\pi^{-1}(b)|=m!$ for all $b\in\B_m$, 
which also follows from equation (\ref{E:fiber}).
\par
From the remark following Corollary \ref{C:gfprod}, the generating 
functions are
\[
O_{\B}(t)=(1+t)^n\quad\text{and}\quad M_{\B}(t)=nt(1+t^2)^{n-1} .
\]
\end{demo}
\begin{demo}
(Monomials; \cite[Ex. 2]{H1}, \cite[Ex. 2.2.1]{F2}.)
\label{Ex:finset}
Let $\SS$ be the category with $\SS_n=\{[n]\}$, where $[n]=\{1,2,\dots,n\}$ 
(and $[0]=\emptyset$), and let $\Hom([m],[n])$ be the set of injective 
functions from $[m]$ to $[n]$.  Then the axioms are easily seen to hold, 
with $\Aut[n]=\Si_n$, the symmetric group on $n$ letters.
Since $\Hom([n],[n+1])$ has $(n+1)!$ elements, we have $u([n];[n+1])=1$ 
and $d([n];[n+1])=n+1$ (so $\SS$ is factorial).
The generating functions are
\begin{equation}
\label{E:ssgf}
O_{\SS}(t)=\sum_{n\ge 0}\frac{t^n}{n!}=e^t
\quad\text{and}\quad
M_{\SS}(t)=\sum_{n\ge 0}\frac{t^{2n+1}}{n!}=te^{t^2}.
\end{equation}
\par
Now let $\M=\SS^n$.
Objects of $\M$ can be identified with monomials in $n$ commuting
indeterminates $t_1,\dots,t_n$.  The automorphism group of 
$t_1^{i_1}t_2^{i_2}\cdots t_n^{i_n}$ is $\Si_{i_1}\times\Si_{i_2}
\times\dots\times\Si_{i_n}$, and a monomial $u$ precedes a monomial
$v$ in the induced partial order if $u$ is a factor of $v$.
By Theorem \ref{T:fprod} $\M$ is factorial, so
\[
d(1;t_1^{i_1}\cdots t_n^{i_n})=(i_1+\dots+i_n)! 
\]
by equation (\ref{E:dfact}).
Hence by equation (\ref{E:ratio})
$$
u(1;t_1^{i_1}\cdots t_n^{i_n})=\frac{(i_1+\dots+i_n)!}{i_1!\cdots i_n!} .
$$
Then it follows (using equation (\ref{E:coversum})) that
\[
|\tilde{\M^\up}_m|=\sum_{i_1+\dots+i_n=m}\binom{m}{i_1\ i_2\ \cdots i_n}=n^m
\]
and
\[
|\tilde{\M^\dn}_m|=\sum_{i_1+\dots+i_n=m}m!=n(n+1)\cdots (n+m-1).
\]
\par
The weighted-relation poset $Wrp(\M^\up)$ appears in \cite{H1},
where it is shown that the universal cover $\tilde{\M^\up}$ can be
identified with the simple updown category whose objects are
monomials in $n$ noncommuting indeterminates $T_1,\dots,T_n$, and
whose morphisms are inclusions as left factors; the covering
map $\pi:\tilde{\M^\up}\to\M^\up$ sends $T_i$ to $t_i$ (e.g.,
$\pi^{-1}(t_1^2t_2)=\{T_1^2T_2,T_1T_2T_1,T_2T_1^2\}$).  
\par
A similar description of $\tilde{\M^\dn}$ can be obtained by
reworking the construction of Theorem \ref{T:ucover} as follows.
Objects in $\tilde{\M^\dn}$ are those monomials in
the noncommuting indeterminates $\{T_{ij}: 1\le i\le n, j\ge 1\}$
such that (a) no indeterminate is repeated; and (b) if $T_{ij}$
occurs, then so does $T_{ik}$ for $k<j$.  The covering map
$\pi:\tilde{\M^\dn}\to\M^\dn$ sends $T_{ij}$ to $t_i$ (e.g.,
$\pi^{-1}(t_1^2t_2)=\{T_{11}T_{12}T_{21}, T_{12}T_{11}T_{21},
T_{11}T_{21}T_{12}, T_{12}T_{21}T_{11}, T_{21}T_{11}T_{12}, 
T_{21}T_{12}T_{11}\}$).  
For any object $w$ of $\tilde{\M^\dn}$, there are $n$ permutations $\si_1,
\si_2\dots,\si_n$ that can be extracted from the second subscripts:
e.g., for $T_{13}T_{21}T_{11}T_{12}$ the permutations are $\si_1=312$
and $\si_2=1$.  The partial order on objects of $\tilde{\M^\dn}$
is given by having the monomial $wT_{ij}$ cover $w'$ if $w'$ has
the same sequence of first subscripts as $w$, and $wT_{ij}$ has
the same associated permutations as $w'$ except that $\si_i$ for $wT_{ij}$
covers $\si_i$ for $w'$ in the sense of the preceding example.  For
example, $T_{13}T_{21}T_{11}T_{12}$ generates the order ideal
$
T_{13}T_{21}T_{11}T_{12}\rhd T_{12}T_{21}T_{11}\rhd T_{11}T_{21}\rhd
T_{11}\rhd 1 .
$
\par
By equations (\ref{E:ssgf}) and the remark following Corollary
\ref{C:gfprod}, the generating functions are
\[
O_{\M}(t)=e^{nt}\quad\text{and}\quad M_{\M}(t)=nte^{nt^2} .
\]
\end{demo}
\begin{demo}\rm
\label{Ex:fingraph}
Let $\G$ be the category whose objects are isomorphism classes of finite
graphs.
Then $\G$ is graded by the number of vertices, with $\hat 0$ the
empty graph.  A morphism from $H$ to $G$ is an injective function
$f:v(H)\to v(G)$ on the vertex sets such that $f(v_1)$ and $f(v_2)$ are
connected in $G$ if and only if $v_1$ and $v_2$ are connected in
$H$.  If $G\rhd H$, then there is a vertex $v$ of $G$ so that $G-\{v\}$
is isomorphic to $H$.  
Evidently $\G$ is factorial, since for any $G\in\G_n$
\[
\sum_{|H|=n-1} d(H;G) = n .
\]
But $\G$ is also uniformly up-covered, any $G$ covering $H\in\G_n$ 
can be obtained from $H$ by adjoining a new vertex and edges between
that vertex and some subset of the $n$ vertices of $\H$:  thus
\[
\sum_{|G|=n+1} u(H;G) = 2^n .
\]
It follows from Corollary \ref{C:eucedc} that
\[
O_{\G}(t)=\sum_{n\ge 0} \frac{2^{\binom{n}{2}}}{n!}t^n ,
\]
and thus from Theorem \ref{T:gfsum} that
\[
M_{\G}(t)=\sum_{n\ge 1} \frac{2^{\binom{n}{2}}}{(n-1)!}t^{2n-1} .
\]
\par
Objects of the universal cover $\tilde\G^{\up}$ can be identified
with graphs whose vertices are labelled by the positive integers;
morphisms of $\G^{\up}$ preserve labels.  
From equation (\ref{E:covereuc}) follows $|\tilde\G_n^\up|=2^{\binom{n}{2}}$.
On the other hand, an element
$(\emptyset,G_1,G_2,\dots,G_n)$ of $\tilde\G_n^{\dn}$ 
can be specified by giving a bijection
$$
f:\{1,2,\dots,n\}\to v(G_n)
$$
such that each $G_i$ is the full subgraph of $G_n$ on the vertices 
$\{f(1),\dots,f(i)\}$.  This makes it evident that $|\pi^{-1}(G_n)|=n!$,
in accordance with the remark following Theorem \ref{T:ucover}.
\end{demo}
\begin{demo}\rm
\label{Ex:neck}
(Necklaces; \cite[Ex. 3]{H1}.)
For a fixed positive integer $c$, let $\Ne_m$ be the set of 
$m$-bead necklaces with beads of $c$ possible colors.  More
precisely, a rank-$m$ object of $\Ne$ is an equivalence
class of functions $f:\Zm\to [c]$, where $f$ is equivalent
to $g$ if there is some $n$ so that $f(a+n)=g(a)$ for all
$a\in\Zm$.  Thus, for $c=2$ the equivalence class 
$$
\{(1,1,2,2),(2,1,1,2),(2,2,1,1),(1,2,2,1)\}\quad
\text{represents the necklace}
\hskip .3in
\psdots[dotstyle=*](.25,.15)(0,.4)
\psdots[dotstyle=o](-.25,.15)(0,-.1)
\psarc(0,.15){.25}{10}{80}
\psarc(0,.15){.25}{100}{170}
\psarc(0,.15){.25}{190}{260}
\psarc(0,.15){.25}{280}{350}
\hskip .2in .
$$
\vskip .1in
\par\noindent
A morphism from the equivalence class of $f$ in $\Ne_m$ to 
the equivalence class of $g$ in $\Ne_n$ is an injective function 
$h:\Zm\to\Zn$ with $f(a)=gh(a)$ for all $a\in\Zm$, and such that 
$h$ preserves the cyclic order, i.e., if we pick representatives of 
the $h(i)$ in $\mathbf Z$ with $0\le h(i)\le n-1$, then some cyclic 
permutation of $(h(0),h(1),\dots,h(m-1))$ is an increasing sequence.
Informally, $u(p;q)$ is the number of ways to insert a bead
into necklace $p$ to get necklace $q$, and $d(p;q)$ is the
number of different beads of $q$ that can be deleted to give $p$.
\par
Note that $\Ne$ is factorial (there are $m$ different beads that 
can be removed from $p\in\Ne_m$) and also evenly up-covered with
$u_m=mc$ for $m\ge 1$ (in a necklace with $m\ge 1$ beads there
are $m$ places that a bead of $c$ possible colors can be inserted);
of course $u_0=c$.
Thus, by Corollary \ref{C:eucedc}
$$
a_n=\begin{cases} 1,&\text{if $n=0$;}\\
\frac{c^n}{n},&\text{if $n\ge 1$;}\end{cases}
$$
and so $O_{\Ne}(t)=1-\log(1-ct)$.
Again using the fact that $\Ne$ is factorial 
(and Theorem \ref{T:gfsum}), we have
\[
M_{\Ne}(t)= \frac{ct}{1-ct^2} .
\]
\par
We have $|\tilde\Ne_m^\up|=(m-1)!c^m$ by equation (\ref{E:covereuc}):  cf.
the discussion in \cite{H1}, where the same result is obtained by identifying
elements of rank $\tilde\Ne_m^\up$ with necklaces of $m$ beads in $c$ colors 
in which the beads are labelled by $1,2\dots,m$.
On the other hand, an element of $\tilde\Ne_m^\dn$
can be regarded as an equivalence class of pairs $(f,\si)$, 
where $f:\Zm\to [c]$ and $\si$ is a permutation of $\{0,1,\dots,m-1\}$.
The equivalence relation is that $(f,\si)\sim (g,\tau)$ if $f\ne g$ and
there is some $0\le n\le m-1$ with $g(x)=f(x+n)$ and $\tau(x)=\si(x+n)$ for 
all $x\in\Zm$.  
Evidently there are $m!$ such equivalence classes for a
given $[f]\in\Ne_m$, in accord with the factoriality of $\Ne$.
\end{demo}
\begin{demo}\rm
\label{Ex:upart}
(Integer partitions with unit weights; \cite[Ex. 5]{H1}, \cite{S1},
\cite[Ex. 1.6.8]{F2}.)
Let $\Y$ be the category with $\Ob\Y$ the set of integer partitions,
i.e., finite sequences $(\la_1,\la_2,\dots,\la_k)$ of positive integers 
with $\la_1\ge\la_2\ge\dots\ge\la_k$.
The rank of a partition is $|\la|=\la_1+\la_2+\dots+\la_k$; we write
$\ell(\la)$ for the length (number of parts) of $\la$.  The set of morphisms
$\Hom(\la,\mu)$ contains a single element if and only if $\la_i\le\mu_i$
for all $i$.
Then $\Y$ is evidently unital but not simple.
\par
Since $\Y_n$ is the set of partitions of $n$, the object generating function
$$
O_{\Y}(t)=\sum_{n\ge 0}|\Y_n|t^n=\frac1{(1-t)(1-t^2)(1-t^3)\cdots}
$$
is familiar.  The morphism generating function is
$$
M_{\Y}(t)=\sum_{n\ge 0}|\{(\la,\mu):\: \la\in\Y_n,\: \la\lhd\mu\}| t^{2n+1} 
$$
since $\Y$ is unital.  Using the case $k=1$ of \cite[Theorem 3.2]{S1},
it follows that
$$
M_{\Y}(t)=\frac{t}{1-t^2}O_{\Y}(t^2)=
\frac{t}{(1-t^2)^2(1-t^4)(1-t^6)\cdots} .
$$
\par
In \cite{H1} it is shown that the universal cover $\tilde\Y$ is the 
poset of standard Young tableaux, so $u(\hat 0;\la)=d(\hat 0;\la)$ is
the number of standard Young tableaux of shape $\la$.
\end{demo}
\begin{demo}\rm
\label{Ex:kpart}
Let $\K$ be the category with $\Ob\K$ the set of
integer partitions, and $\Hom(\la,\mu)$ defined as follows.
Let $\la=(\la_1,\dots,\la_n)$ and $\mu=(\mu_1,\dots,\mu_m)$,
always written in decreasing order.
Then a morphism from $\la$ to $\mu$ is an injective
function $f:[n]\to[m]$ such that $\la_i\le \mu_j$ whenever $f(i)=j$.
\par
The partial order induced on $\Ob\K=\Ob\Y$ is the same as that of the
preceding example:  the difference is that we now have nontrivial
automorphism groups and weights on covering relations.
The automorphism group of $\la=(\la_1,\dots,\la_k)$ is the
subgroup of $\Si_k$ consisting of those permutations $\si$
such that $\la_i=\la_j$ whenever $\si(i)=j$.  If we let $m_i(\la)$ 
be the number of parts of $\la$ of size $i$, this means that
$$
|\Aut(\la)|=m_1(\la)!m_2(\la)!\cdots .
$$
For partitions $\la,\mu$ with $|\mu|=|\la|+1$, $\Hom(\la,\mu)$
is nonempty exactly when (i) $\mu$ comes from $\la$ by adding a
part of size 1; or (ii) $\mu$ comes by replacing a size-$k$ part
of $\la$ by a part of size $k+1$.  In case (i) we put $u(\la;\mu)=1$
and $d(\la;\mu)=m_1(\mu)$, while in case (ii) we put $u(\la;\mu)=
m_k(\la)$ and $d(\la;\mu)=m_{k+1}(\mu)$.
The weights $d(\la;\mu)$ appear implicitly in \cite{Ki} and
explicitly in \cite{K}, where they are referred to as 
``Kingman's branching'':  see especially Figure 4 of \cite{K}.
\par
The object generating function can be computed as follows:
\begin{multline*}
O_{\K}(t)=\sum_{\la\in\Ob\K}\frac{t^{|\la|}}{|\Aut(\la)|}
=\sum_{m_1,m_2,\dots \ge 0}\frac{t^{m_1+2m_2+\dots}}{m_1!m_2!\cdots}=\\
\left(\sum_{m_1\ge 0}\frac{t^{m_1}}{m_1!}\right)
\left(\sum_{m_2\ge 0}\frac{t^{2m_2}}{m_2!}\right)\cdots
=\exp(t+t^2+\cdots)=\exp\left(\frac{t}{1-t}\right) .
\end{multline*}
To find the morphism generating function
\begin{equation}
\label{E:morkpart}
M_{\K}(t)=\sum_{\la\in\Ob\K}\frac{t^{2|\la|+1}}{|\Aut(\la)|}
\sum_{\la\lhd\mu}u(\la;\mu) 
\end{equation}
we first observe that
$$
\sum_{\la\lhd\mu}u(\la;\mu) = 1+\ell(\la)=1+m_1(\la)+m_2(\la)+\cdots ,
$$
since
(using the description of $u(\la;\mu)$ above) this is the number of ways
to obtain a partition covering $\la$:  we can increase by one any of the 
$\ell(\la)$ parts of $\la$, or add a new part of size 1.  Thus equation
(\ref{E:morkpart}) is
\begin{multline*}
M_{\K}(t)=\sum_{m_1,m_2,\dots\ge 0}\frac{t^{1+2m_1+4m_2+\cdots}}
{m_1!m_2!\cdots}(1+m_1+m_2+\cdots)\\
=(t+t^3+t^5+\cdots)O_{\K}(t^2)
=\frac{t}{1-t^2}\exp\left(\frac{t^2}{1-t^2}\right) .
\end{multline*}
\par
The universal cover $\tilde\K^\up$ can be described in terms
of set partitions:  elements of $\tilde\K_n^\up$ are partitions
of the set $[n]$, with $\pi:\tilde\K^\up\to\K^\up$ sending each partition
to the integer partition of $n$ given by its block sizes.
Thus $|\tilde\K_n^\up|$ is the $n$th Bell number \cite[A000110]{S}.
We can identify set partitions with the construction of Theorem
\ref{T:ucover} as follows.
For convenience we write a set partition as
$(P_1,\dots,P_k)$ with $|P_1|\ge |P_2|\ge\dots\ge|P_k|$
and, if $|P_i|=|P_j|$ for $i<j$, then $\max P_i<\max P_j$.
Assign the unique partition of $[1]$ to the morphism from $\hat 0$ to $(1)$,
and suppose inductively that we have assigned an ordered partition
$P=(P_1,\dots,P_k)$ of $[n]$ to the chain $(h_1,\dots,h_n)$ of morphisms 
between adjacent ranks of $\K^\up$ from $\hat 0$ to 
$\trg(h_n)=(\la_1,\dots,\la_k)\in\Ob\K_n^\up$ so that $\la_i=|P_i|$.
Let $f\in\Hom_{\K}(\la,\mu)$ be a representative of the equivalence
class $h_{n+1}\in\Hom_{\K^\up}(\la,\mu)$, where $|\mu|=n+1$.
If $\mu$ has length $k+1$, assign $(P_1,\dots,P_k,\{n+1\})$ to the chain
$(h_1,\dots,h_n,h_{n+1})$.  Otherwise, $\mu$ has length $k$ and there is
a unique $i\in [k]$ such that $\la_i<\mu_{f(i)}$:  in this case,
assign to $(h_1,\dots,h_{n+1})$ the rearrangement of $(P_1',\dots,P_k')$, 
where
$$
P_j'=\begin{cases} P_j\cup\{n+1\},&\text{if $j=i$},\\
P_j,&\text{otherwise,}
\end{cases}
$$
so that $P_i'$ immediately follows $P_m'$, where 
$m=\max\{j<i:|P_j'|\ge |P_i'|\}$.
Evidently the set partition assigned to $(h_1,\dots,h_{n+1})$ projects
to $\mu$ in either case.
\par
Rank-$n$ objects of the universal cover $\tilde\K^\dn$ can be described
as sequences $s=(a_1,\dots,a_n)$ such that $m_1(s)\ge m_2(s)\ge\cdots$,
where $m_i(s)$ is the number of occurrences of $i$ in $s$; the covering
map sends $s$ to $(m_1(s),m_2(s),\dots)$.
See \cite[A005651]{S}.
As in the preceding paragraph,
we can proceed inductively to identify these objects with the 
construction of Theorem \ref{T:ucover}.
Start by assigning $s=(1)$ to the morphism from $\hat 0$ to $(1)$.
Suppose now we have assigned $s=(a_1,\dots,a_n)$ to a chain of morphisms
$(h_1,\dots,h_n)$ between adjacent ranks of $\K^\dn$ from $\hat 0$ to 
$\la=(\la_1,\dots,\la_k)\in\K_n^\dn$ so that $m_i(s)=\la_i$ for 
$1\le i\le k$, and let $h_{n+1}\in\Hom_{\K^\dn}(\la,\mu)$ where $|\mu|=n+1$.  
Now a representative  $f\in\Hom_{\K}(\la,\mu)$ of $h_{n+1}$ must
be ``almost an automorphism'' exchanging parts of equal size with
just one exception:  there is a unique $i\in [\ell(\mu)]$ such
that either $i$ is not in the image of $f$ (in which case $\mu_i=1$),
or else $\la_{f^{-1}(i)}<\mu_i$ (in which case $\mu_i=\la_{f^{-1}(i)}+1$).
Let $S=\{j > i: \la_{f^{-1}(j)}=\mu_i\}$:  note that $S$ is
independent of the choice of $f$.  Now define a permutation $\si$
of $[\ell(\mu)]$ as follows.  If $S=\emptyset$, let $\si$ be the
identity; otherwise, if $S=\{i+1,\dots,l\}$, let $\si(a)=a+1$ for
$i\le a\le l-1$, $\si(l)=i$, and $\si(a)=a$ for $a\notin\{i,\dots,l\}$.
We then assign the sequence $s'=(\si(a_1),\dots,\si(a_k),i)$ to
the chain $(h_1,\dots,h_n,h_{n+1})$.  If $i\notin\im f$, then
$\mu_j=1$ for $j\ge i$ and either $i=\ell(\mu)=k+1$ (if $S$ is empty)
or $l=\ell(\mu)=k+1$ (if it isn't):  either way $\mu$ differs by $\la$
by having 1 inserted in the $i$th position, and $s'$ projects to $\mu$.
If $\mu_i=\la_{f^{-1}(i)}+1$, then we must have $\la_j=\mu_j$ for
$j<i$, and $\mu$ differs from $\la$ in having a part of size $\mu_i-1$
increased by 1.  If $S$ is empty, $\la_{f^{-1}(i)}=\la_i$ and 
$\mu_i=m_i(s')=m_i(s)+1=\la_i+1$.  
Otherwise, $\mu_i=m_i(s')=m_l(s)+1=\la_{f^{-1}(l)}$
and $m_{j+1}(s')=m_j(s)$ for $i\le j\le l-1$.  Either way, $s'$
again projects to $\mu$.
\end{demo}
\begin{demo}\rm
(Integer compositions; \cite[Ex. 6]{H1}.)
\label{Ex:comp}
Let $\C_n$ be the set of integer compositions of $n$, i.e. sequences 
$I=(i_1,\dots,i_p)$ of positive integers with $a_1+\dots+a_m=n$;
as with partitions we write $\ell(I)$ for the length of $I$.
A morphism from $(i_1,\dots,i_p)\in\C_n$ to 
$(j_1,\dots,j_q)\in\C_m$ is an order-preserving injective
function $f:[p]\to[q]$ such that $i_a\le j_{f(a)}$ for all
$a\in [p]$.  Then $\C$ is a univalent updown category (but
not simple).  
\par
The object generating function is
$$
O_{\C}(t)=\sum_{n\ge 0}|\C_n|t^n=1+\sum_{n\ge 1}2^{n-1}t^n=\frac{1-t}{1-2t} .
$$
Now for any composition $I$, 
$$
\sum_{I\lhd J} u(I;J)=\ell(I)+\ell(I)+1=2\ell(I)+1
$$
since we can get a composition covering $I$ either by increasing each
of its $\ell(I)$ parts, or by inserting a part of size 1 into one
of $\ell(I)+1$ possible positions.
Thus, the morphism generating function is
$$
M_{\C}(t)=\sum_{n\ge 0}t^{2n+1}\sum_{k=1}^n|\C_{n,k}|(2k+1) ,
$$
where $\C_{n,k}$ is the set of compositions of $n$ with $k$ parts.
Evidently $|\C_{n,k}|=\binom{n-1}{k-1}$, so
$$
M_{\C}(t)=\sum_{n\ge 0}t^{2n+1}\sum_{k=1}^n\binom{n-1}{k-1}(2k+1)=
\sum_{n\ge 0}(n+2)2^{n-1}t^{2n+1}=\frac{t-t^3}{(1-2t^2)^2} .
$$
\par
The universal cover $\tilde\C$ is constructed in \cite{H1} using 
Cayley permutations as defined in \cite{MF}:  
a Cayley permutation of rank $n$ is a length-$n$ sequence 
$s=(a_1,\dots,a_n)$ of positive integers such that any positive 
integer $i<j$ appears in $s$ whenever $j$ does.  
See \cite[A00679]{S}.
The covering map $\pi:\tilde\C\to\C$ sends a sequence $s$ to the 
composition $(m_1(s),m_2(s),\dots)$.
To relate this to the construction of Theorem \ref{T:ucover},
we again proceed inductively.  Send the morphism from $\hat 0$
to $(1)$ to the Cayley permutation $(1)$, and suppose 
we have assigned to a chain $(h_1,h_2,\dots,h_n)$ of morphisms
between consecutive ranks of $\C$ from $\hat 0$ to $I=(i_1,\dots,i_k)
\in\C_n$ a Cayley permutation $s=(a_1,\dots,a_n)$ that projects to $I$:
note that $\max\{a_1,\dots,a_n\}=k$.
Now let $h_{n+1}\in\Hom(I,J)$ with $J\in\C_{n+1}$.
Then either $\ell(J)=k$ and $h_{n+1}$ is the identity function on $[k]$, 
or $\ell(J)=k+1$.
In the first case, there is exactly one position $q$ where $J$
differs from $I$:  assign to $(h_1,\dots,h_{n+1})$ the Cayley
permutation $s'=(a_1,\dots,a_n,q)$.  Then $m_q(s')=m_q(s)+1=i_q+1$
and $m_i(s')=m_i(s)$ for $i\ne q$, so $s'$ projects to $J$.
In the second case, there is exactly one element $q\in[k+1]$ that 
$h_{n+1}$ misses:  assign $s'=(h_{n+1}(a_1),\dots,h_{n+1}(a_n),q)$ 
to $(h_1,\dots,h_{n+1})$.  Then $\pi(s')=(m_1(s'),m_2(s'),\dots)$ differs
from $I$ only in having an additional 1 inserted in the $q$th place,
and so must be $J$.
\end{demo}
\begin{demo}\rm
(Planar rooted trees; \cite[Ex. 4]{H1}.)
\label{Ex:plrtree}
Let $\P_n$ consist of functions $f:[2n]\to\{-1,1\}$ so that the 
partial sums $S_i=f(1)+\cdots+f(i)$ have the properties that $S_i\ge 0$
for all $1\le i\le 2n$, and $S_{2n}=0$.  We declare $\Aut(f)$ to
be trivial for all objects $f$ of $\P$, and define a morphism from
$f\in\P_n$ to $g\in\P_{n+1}$ to be an injective, order-preserving function 
$h:[2n]\to[2n+2]$ such that the two values of $[2n+2]$ not in the
image of $h$ are consecutive, and $f(i)=gh(i)$ for $1\le i\le 2n$.
Then $\P$ is a univalent updown category.  Using the well-known
identification of balanced bracket arrangements with planar
rooted trees, e.g. 
\begin{equation*}
(1,1,-1,1,1,-1,-1,-1)\quad\text{is identified with}\quad
\psline{*-*}(.25,0)(.5,.5)
\psline{*-*}(.5,.5)(.75,0)
\psline{*-*}(.75,0)(.75,-.5)
\hskip .4in ,
\end{equation*}
\vskip .2in
\par\noindent
we can think of $\P$ as the updown category of planar rooted trees;
the rank is the count of non-root vertices.  (The empty bracket
arrangement $\emptyset$ is identified with the tree $\bullet$ consisting 
of the root vertex.)
In view of the well-known enumeration of planar rooted trees by
Catalan numbers, the object generating function is simply
$$
O_{\P}(t)=\sum_{n\ge 0}|\P_n|t^n=\sum_{n\ge 0}\frac1{n+1}\binom{2n}{n}t^n=
\frac{1-\sqrt{1-4t}}{2t} .
$$
Since there are $2n+1$ possibilities for order-preserving injections
$[2n]\to[2n+2]$ that miss two consecutive values, $\P$ is evenly
up-covered with $u_n=2n+1$ and by Theorem \ref{T:gfsum}
the morphism generating function is
$$
M_{\P}(t)=\sum_{n\ge 0}\frac{2n+1}{n+1}\binom{2n}{n}t^{2n+1}=
\sum_{n\ge 0}\binom{2n+1}{n+1}t^{2n+1}=
\frac{1-\sqrt{1-4t^2}}{2t\sqrt{1-4t^2}} .
$$
\par
From equation (\ref{E:covereuc}) we have $|\tilde\P_n|=(2n-1)!!$ for
$n\ge 1$.
In \cite{H1} the universal cover of $Wrp(\P)$ is described as the 
weighted-relation poset whose rank-$n$ elements are permutations 
$(a_1,a_2,\dots,a_{2n})$ of the multiset $\{1,1,2,2,\dots,n,n\}$ such that, 
if $a_i>a_j$ with $i<j$, then there is some $k<j$, $k\ne i$, such that
$a_k=a_i$. (The covering map sends a sequence $s=(a_1,\dots,a_{2n})$
to a sequence of 1's and $-1$'s by sending the first occurrence
of $i$ in $s$ to 1 and the second to $-1$.)  
This construction can be identified with $\tilde\P$ as described in 
Theorem \ref{T:ucover} in an obvious way.
For example, the morphism from $\emptyset$
to $(1,1,-1,1,-1,-1)$ given by the composition $h_3h_2h_1$,
where $h_1=\emptyset$, $h_2=\{(1,1),(2,4)\}$ and 
$h_3=\{(1,1),(2,2),(3,3),(4,6)\}$, can be coded by the sequence 
$(1,2,2,3,3,1)$.
\end{demo}
\begin{demo}\rm
(Rooted trees; \cite[Ex. 7]{H1}.)
\label{Ex:rtree}
Let $\T_n$ consist of partially ordered sets $P$ such that
(i) $P$ has $n+1$ elements; (ii) $P$ has a greatest element; and
(iii) for any $v\in P$, the set of elements of $P$ exceeding $v$
forms a chain.
The Hasse diagram of such a poset $P$ is a tree with the greatest 
element (the root vertex) at the top.
A morphism of $\T$ from $P\in\T_m$ to $Q\in\T_n$ is an injective
order-preserving function $f:P\to Q$ that sends the root of $P$ to the
root of $Q$, and which preserves covering relations (i.e., if
$v\lhd w$ in the partial order on $P$, then $f(v)\lhd f(w)$ in
the partial order on $Q$).  Then $\T$ is an updown category.
\par
The updown category $\T$ was studied extensively in \cite{H2},
though without using the categorical language.  To see that the
construction of the preceding paragraph gives the same multiplicities
as in \cite{H2}, consider a morphism from $P\in\T_n$ to $Q\in\T_{n+1}$.  
Any such morphism misses only some terminal vertex $v\in Q$, so we
can think of it as identifying $P$ with $Q-\{v\}$.  Elements of
$$
\Hom(P,Q)/\Aut(Q)
$$
amount to different choices for the parent of $v$ in $Q$, i.e.,
different choices for terminal vertices of $P$ to which a new
edge and vertex can be attached to form $Q$:  this is $n(P;Q)$
as defined in \cite{H2}.  On the other hand, elements of 
$$
\Hom(P,Q)/\Aut(P)
$$
amount to different choices of $v$, and thus to different
choices for an edge of $Q$ that when cut leaves $P$:  this is
$m(P;Q)$ as defined in \cite{H2}.  
\par
The object generating function
$$
O_{\T}(t)=\sum_{n\ge 0}t^n\sum_{P\in\T_n}\frac1{|\Aut(P)|}
$$
can be evaluated using a result of \cite{B}.  First, we note from
\cite{BK} (cf. the discussion in \cite{H1}) that 
$$
u(\bullet;P)=n(\bullet;P)=\frac{(|P|+1)!}{P!|\Aut(P)|} ,
$$
where $P!$ is the ``tree factorial,'' i.e., the product
$$
\prod_{v\ \text{is a vertex of $P$}}(|P_v|+1)!
$$
where $P_v$ is the subtree of $P$ having $v$ as its root.  Thus
$$
O_{\T}(t)=\sum_{n\ge 0}t^n\sum_{P\in\T_n}\frac{u(\bullet;P)P!}{(n+1)!} .
$$
From \S5.3 of \cite{B} we have
$$
\sum_{P\in\T_n}u(\bullet;P)P!=(n+1)^n ,
$$
so
$$
O_{\T}(t)=\sum_{n\ge 0}\frac{(n+1)^n}{(n+1)!}t^n .
$$
(We note that $tO_{\T}(t)$ is the functional inverse of $te^{-t}$:
see \cite[\S5.3]{S3}.)
Now $P\in\T_n$ has a total of $n+1$ vertices to which new
edges can be added, so $\T$ is evenly up-covered with $u_n=n+1$
and by Theorem \ref{T:gfsum} the morphism generating function is
$$
M_{\T}(t)=\sum_{n\ge 0} \frac{(n+1)^n}{n!}t^{2n+1} .
$$
\par
In \cite{H1} the weighted-relation poset $Wrp(\T^\up)$ is discussed,
and it is shown that rank-$n$ objects of the universal cover $\tilde\T^\up$ 
can be described as permutations of $[n]$.  (The partial order on 
permutations in $\tilde\T^\up$ is as follows:  a permutation $\tau$ of $[n+1]$ 
covers the permutation $\tau\iota_{\tau^{-1}(n+1)}^n$ of $[n]$,
where $\iota_m^n$ is the order-preserving injection from $[n]$ to $[n+1]$
that misses $m$.)
On the other hand, objects of $\tilde\T_n^\dn$ can be thought of as 
pairs $(P,f)$, where $P\in\T_n$ and 
$$
f:\{0,1,2,\dots,n\}\to P
$$
is a bijection such that $f(i)$ exceeds $f(j)$ (in the partial 
order on $P$) whenever $i>j$. (Cf. the remark following \cite[Prop. 2.5]{H2}.)
\end{demo}
\begin{demo}\rm
(Binary rooted trees) 
\label{Ex:brtree}
Let $\B_n$ be the set of rooted binary trees
with $n+1$ terminal vertices (leaves).  That is, an element of $\B_n$
is a rooted tree in which each vertex has two daughters or none (in
which case it is a leaf).
Any $P\in\B_n$ defines a metric on its set $L(P)$ of leaves:  the distance
$\de(p,q)$ from leaf $p$ to leaf $q$ is the number of non-terminal vertices 
contained in the unique shortest path from $p$ to $q$.
An automorphism of $P\in\B_n$ is a bijection $f$ on $L(P)$
such that $\de(f(p),f(q))=\de(p,q)$ for all $p,q$.
A morphism from $P\in\B_n$ to $Q\in\B_{n+1}$ 
is an injection $f:L(P)\to L(Q)$ such that
$\de(f(p_1),f(p_2))\ge \de(p_1,p_2)$ for all $p_1,p_2\in L(P)$, and 
the only $r\in L(Q)$ with $r\notin\im f$ is distance 1 from some 
$s\in\im f$.  For example, if
\[
P=
\psline{*-*}(.25,0)(.5,.5)
\psline{*-*}(.5,.5)(.75,0)
\psline{*-*}(.75,0)(.5,-.5)
\psline{*-*}(.75,0)(1,-.5)
\hskip .5in
\text{and}
\hskip .3in
Q=
\psline{*-*}(.2,0)(.5,.5)
\psline{*-*}(.5,.5)(.8,0)
\psline{*-*}(.8,0)(.6,-.5)
\psline{*-*}(.8,0)(1,-.5)
\psline{*-*}(.2,0)(0,-.5)
\psline{*-*}(.2,0)(.4,-.5)
\]
\vskip .2in
\par\noindent
then $|\Aut(P)|=2$, $|\Aut(Q)|=8$, and $|\Hom(P,Q)|=8$:  hence
$u(P;Q)=1$ and $d(P;Q)=4$.  (If we call a pair of leaves distance
1 apart together with their common parent a ``bud'', then
$u(P;Q)$ is the number of leaves of $P$ that can be replaced by a
bud to get $Q$, and $d(P;Q)$ is twice the number of buds of $Q$
that can be replaced by a leaf to get $P$.)
\par
Let $R\in\B_n$, and let $T$ be a particular realization of $R$ in
the plane, i.e., a planar binary rooted tree.  Since $T$ has $n$ 
non-terminal vertices, the group $G=\Z_2^n$  
(where $\Z_2$ is the group of order 2) acts on $T$ by rotations
around each such vertex:  the isotropy group of $T$ is $\Aut R$.  Then 
the number of distinct planar binary rooted trees $T$ that 
can represent $R$ is
\[
\left| G/\Aut R\right|=\frac{2^n}{|\Aut R|} .
\]
Now there are $C_n$ distinct planar binary rooted trees with $n$
non-terminal vertices, where $C_n$ is the $n$th Catalan number, so
\[
\sum_{P\in\B_n}\frac{2^n}{|\Aut R|} = C_n
\]
and the object generating function is
\[
O_{\B}(t)=
\sum_{R\in\Ob\B}\frac{t^{|R|}}{|\Aut R|}=\sum_{n\ge 0} \frac{C_n}{2^n}t^n
=\frac{1-\sqrt{1-2t}}{t} .
\]
Since $\B$ is evenly up-covered with $u_n=n+1$, by Theorem \ref{T:gfsum}
the morphism generating function is
\[
M_{\B}(t)=
\sum_{n\ge 0} \frac{C_n (n+1)}{2^n} t^{2n+1} = \frac{t}{\sqrt{1-2t^2}} .
\]
\end{demo}
\par
Since $\B$ is evenly up-covered, equation (\ref{E:covereuc})
implies $|\tilde\B_n^{\up}|=n!$, and by Theorem
\ref{T:unique} $\tilde\B^{\up}$ must be isomorphic to $\tilde\T^{\up}$.
In fact, there is a natural way to associate a permutation
of $[n]$ to any $c\in\tilde\B_n^{\up}$.  
Given $(c_0,c_1,\dots,c_n)\in\tilde\B_n^{\up}$, 
there is a corresponding planar binary root tree with labelled 
non-terminal vertices:  a node gets label $i$ if $c_{i-1}\to c_i$
involves adding a bud at that node.  Put another set of labels 
$0,1,\dots, n$ on the leaves, running left to right.
For example, the two elements of $\tilde\B_2^{\up}$ are
\vskip .2in
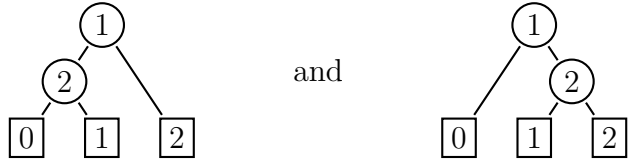
\[
\cnodeput(0,.75){A}{1}
\cnodeput(-.5,0){B}{2}
\rput(1,-.75){\rnode{C}{\psframebox{2}}}
\rput(-1,-.75){\rnode{D}{\psframebox{0}}}
\rput(0,-.75){\rnode{E}{\psframebox{1}}}
\psset{nodesep=1pt}
\ncline{-}{A}{B}
\ncline{-}{A}{C}
\ncline{-}{B}{D}
\ncline{-}{B}{E}
\hskip 1in
\text{and}
\hskip 1in
\cnodeput(0,.75){A}{1}
\cnodeput(.5,0){B}{2}
\rput(-1,-.75){\rnode{C}{\psframebox{0}}}
\rput(0,-.75){\rnode{D}{\psframebox{1}}}
\rput(1,-.75){\rnode{E}{\psframebox{2}}}
\psset{nodesep=1pt}
\ncline{-}{A}{B}
\ncline{-}{A}{C}
\ncline{-}{B}{D}
\ncline{-}{B}{E}
\hskip 1in .
\]
\vskip .4in
\par\noindent
Now define a permutation of $[n]$ by sending $i\in[n]$ to the label on the
last common ancestor of the leaves labelled $i-1$ and $i$.  For example,
our two labelled trees above correspond respectively to the permutation
exchanging 1 and 2, and to the identity permutation.
\par
An element $U\in\tilde\B_n^\dn$ with $\pi(U)=V\in\B_n^\dn$ can be thought of as
$V$ equipped with an appropriate set of labels on its edges so that
exactly one of each pair of edges coming out of a non-terminal vertex
carries a label.  More precisely, let $\hat V$ be the set of 
non-terminal vertices of $V$:  then $U\in\pi^{-1}(V)$ can be identified
with a pair of functions $(g,h)$, where $g:[n]\to\hat V$ is a bijection
such that $g(v)>g(w)$ in $U$ when $v>w$, and $h:\hat V\to\{L,R\}$
(so that there are $2^n$ possiblities for $h$).  In this way one sees,
e.g., that for the binary trees $P,Q$ above one has $|\pi^{-1}(P)|=1\cdot 2^2
=4$ and $|\pi^{-1}(Q)|=2\cdot 2^3=16$.
\par
A summary of our examples (U=univalent, UC=evenly up-covered,
F=factorial):
\vskip .1in
\par\noindent
\begin{tabular}{|l|l|l|l|l|l|l|}\hline
\# & Description & Object g.f. & Morphism g.f. & U & UC & F\\ 
\hline
1 & Subsets of $[n]$ & $(1+t)^n$ & $nt(1+t^2)^{n-1}$ & yes & yes & yes\\ \hline
2 & Monomials & $e^{nt}$ & $nte^{nt^2}$ & no & yes & yes\\ \hline
3 & Finite graphs & $\sum_{n\ge 0}2^{\binom{n}2}\frac{t^n}{n!}$ &
$\sum_{n\ge 1}2^{\binom{n}2}\frac{t^{2n-1}}{(n-1)!}$ & no & yes & yes\\ \hline
4 & Necklaces & $1-\log(1-ct)$ & $\frac{ct}{1-ct^2}$ & no & yes & yes\\ \hline
5 & Partitions & $\prod_{n=1}^\infty\frac1{1-t^n}$  & 
$\frac{t}{1-t^2}\prod_{n=1}^\infty\frac1{1-t^{2n}}$ & yes & no & no\\ \hline
6 & Partitions & $\exp(\frac{t}{1-t})$ & $\frac{t}{1-t^2}\exp(\frac{t^2}
{1-t^2})$ & no & no & no\\ \hline
7 & Compositions & $\frac{1-t}{1-2t}$ & $\frac{t(1-t^2)}{(1-2t^2)^2}$ &
yes & no & no\\ \hline
8 & Planar rtd. trees & $\frac{1-\sqrt{1-4t}}{2t}$ & 
$\frac{1-\sqrt{1-4t^2}}{2t\sqrt{1-4t^2}}$ & yes & yes & no\\ \hline
9 & Rooted trees & $\sum_{n\ge 0}\frac{(n+1)^n}{(n+1)!}t^n$ &
$\sum_{n\ge 0}\frac{(n+1)^n}{n!}t^{2n+1}$ & no & yes & no\\ \hline
10 & Binary rtd. trees & $\frac{1-\sqrt{1-2t}}{t}$ &
$\frac{t}{\sqrt{1-2t^2}}$ & no & yes & no\\ \hline
\end{tabular}

\end{document}